\documentclass[twocolumn]{autart}    

\usepackage{graphicx}
\usepackage{amssymb}  
\usepackage{amsmath}
\usepackage{enumitem}
\usepackage{bm}
\usepackage[justification=centering]{caption}
\usepackage{mathtools}
\usepackage{changepage} 
\mathtoolsset{showonlyrefs}
\usepackage{amsfonts}
\usepackage{dirtytalk}
\usepackage{bigints}
\usepackage{url}

\newcommand{\Mc}{\mathcal{M}}
\newcommand{\Pc}{\mathcal{P}}
\newcommand{\Uc}{\mathcal{U}}
\newcommand{\Lte}{\mathcal{L}_{2e}}
\newcommand{\Lt}{\mathcal{L}_{2}}
\newcommand{\Lo}{\mathcal{L}_{1}}
\usepackage{color}

\newtheorem{theorem}{Theorem}
\newtheorem{assumption}[theorem]{Assumption}
\newtheorem{definition}[theorem]{Definition}
\newtheorem{example}[theorem]{Example}

\begin{document}
\begin{frontmatter}
\runtitle{Insert a suggested running title} 
\title{Antagonistic Control: Foundations, Scalability and Nonlinearity\thanksref{footnoteinfo}} 
\thanks[footnoteinfo]{Corresponding author Sribalaji C. Anand. This research is supported by the Swedish Research Council grant 2024-00185, and by the Knut and Alice Wallenberg Foundation.}
\author[P,KTH]{Sribalaji C. Anand} 
\author[UU]{Andr\'e M.H. Teixeira }
\address[P]{Department of Electrical and Systems Engineering, University of Pennsylvania, United States.}
\address[KTH]{Department of Decision and Control Systems, KTH Royal Institute of Technology, Stockholm, Sweden.}
\address[UU]{Department of Information Technology, Uppsala University, Sweden.\\(email: srca@kth.se, andre.teixeira@it.uu.se)}
%
%
%
\begin{abstract}
This paper studies the worst-case impact of constrained control inputs: an input seeks to maximize the average cost of some outputs, measured in the $L_2$ or $L_1$ norm, while remaining bounded in terms of other outputs. This problem template subsumes classical metrics such as the $H_\infty$ norm and the output-to-output gain, and arises in adversarial control, security assessment, and robust control. For linear time-invariant systems, we provide an exact semi-definite program (SDP) when there is a single constraint, and SDPs computing upper bounds when there are multiple constraints. We derive sufficient conditions, in terms of system zeros and relative degrees, under which the worst-case cost is unbounded, together with a constructive closed-loop modification that removes the unboundedness. From a security standpoint, unbounded values reveal structural limitations in detecting certain attack inputs. For positive systems, we provide scalable formulations whose complexity grows linearly in the state dimension: a scalable SDP for quadratic costs, and an exact linear program for linear costs. The results extend to nonlinear polynomial systems via a sum-of-squares program. We illustrate the results with numerical examples.
\end{abstract}
\end{frontmatter}
\section{Introduction}
In this paper, we consider a continuous-time (CT) dynamical system of the form 
\begin{equation}\label{P_main}
\begin{aligned}
    \dot{x}(t) &= f(x(t),u(t)),\; x(0)=0\\
    y_i(t) &= h_i(x(t), u(t)), \; i\in \{1,2,\dots,N\}
\end{aligned}
\end{equation}
where $x(t) \in \mathbb{R}^n$ is the state, $u(t) \in \mathbb{R}^m$ is the control input, and $y_i \in \mathbb{R}^{p_i}$ is the system output. Here $y_i$ can represent a physical sensor output or a virtual performance loss output. We are interested in solving the \emph{antagonistic control problem}
\begin{equation}\label{eq:main:sup}
\begin{aligned}
    \sup_{u \in \mathcal{U}} \;\; 
    & \frac{1}{\lvert \mathcal{P} \rvert}
    \sum_{i \in \mathcal{P}} \left[ \int_{0}^{\infty} \ell(y_i(t),Q_i) \, dt \right] \\
    \text{s.t.} \;\;
    & \int_{0}^{\infty} \ell(y_j(t),R_j) \, dt \le \delta_j, 
    \quad \forall j \in \mathcal{M},
\end{aligned}
\end{equation}
where $\ell(\cdot,\cdot) \in \mathbb{R}$ is a loss function mapping the system output to a real number, $\mathcal{P}$ and $\mathcal{M}$ denote the sets of outputs influencing the objective function and the constraints, respectively, and $\delta_j > 0, \forall j \in \mathcal{M}$, denotes the constraint threshold. For regularity, we establish that $Q_i \succeq 0, \forall i \in \mathcal{P}$, $R_j \succeq 0, \forall j\in \mathcal{M}$, $\mathcal{P} \cap \mathcal{M} = \emptyset$, $\mathcal{P} \cup \mathcal{M} = \{1,2,\dots,N\}$, $\mathcal{P} \neq \emptyset$, and $\mathcal{M} \neq \emptyset$. The objective denotes the average cost over the outputs in $\Pc$, while the cost of the outputs in $\Mc$ is bounded.

The central observation of this paper is that a single problem template, \eqref{eq:main:sup}, subsumes several well-studied quantities as special cases: if $N=2$, $\vert \Pc \vert = \vert \Mc \vert=1$, $\ell(y,Q) = y^\top Qy$, and $\ell(y,R) = u^\top Ru$, its value is the $H_{\infty}$ norm of an LTI system \cite{petersen2012robust}; other choices of $\ell(\cdot,\cdot)$, $\mathcal{U}$, and the partition $(\mathcal{P}, \mathcal{M})$ recover the output-to-output $\ell_2$-gain~\cite{teixeira2015strategic} and the ratio between the $H_\infty$ norm and the $H_\_$ index~\cite{wang2007worst}. 
Before presenting the contributions, we discuss an application example.
\begin{example}\label{exmp:attack}
Consider a dynamical system (such as power networks) whose performance is governed by a virtual performance loss output $y_p$. The system operator aims to design control input to maintain the norm of the performance loss output close to zero. Let us suppose that the control inputs to the system is subjected to adversarial attacks \cite{mengis2017data} (due to the presence of unsecured networks). Then an alarm is raised to indicate the presence of an attack when the norm of the measured outputs $y_m$ increases beyond a threshold $\delta$ \cite{teixeira2021security}.

Then the worst-case attack policy can be modeled as
\begin{equation}\label{eq:impact}
\sup_{a \in \mathcal{A}} \left\{ \int_{0}^{\infty} y_p(t)^\top y_p(t)dt \Big| \int_{0}^{\infty} y_m(t)^\top y_m(t)dt \leq \delta \right\}
\end{equation}
where the adversary injects an attack signal $a$ to increase the performance loss while not raising an alarm. Thus \eqref{eq:impact} is a variant of the optimization problem \eqref{eq:main:sup} whose value serves as a measure of attack impact. Design strategies can be employed to minimize the attack impact \cite{lipp2016antagonistic}. $\hfill \triangleleft$
\end{example}
\textbf{Contributions:} The contribution of this paper is a unified analysis of \eqref{eq:main:sup} across system classes (LTI, positive LTI, polynomial), cost functions (quadratic, linear), input signal spaces ($\Lte$, $\Lt^0$, $\Lt^{0+}$, $\Lo^{+}$), and number of constraints, summarized in Table~\ref{main:res:tab}. All results follow one pipeline: dualization of \eqref{eq:main:sup}, followed by (cyclo-)dissipativity arguments, yielding a semi-definite program (SDP), linear program (LP), or sum-of-squares (SOS) program. When there is a single constraint, the resulting programs are \emph{exact}; with multiple constraints, they provide upper bounds on the worst-case cost, for which no generic computation method exists in the literature. We highlight three features of the analysis. First, the value of \eqref{eq:main:sup} is highly sensitive to the input signal space: Example~\ref{exmp:QTP} below exhibits a system whose worst-case cost is finite ($\approx 19.7$) over $\Lt^0$ yet \emph{unbounded} over $\Lte$, exposing a structural vulnerability invisible to the standard analysis. We provide sufficient conditions, in terms of system zeros and relative degrees, under which such unboundedness occurs (Theorems~\ref{thm:inf:zero}, \ref{thm:LTI:lower}, \ref{thm:LTI:L2}), together with a constructive modification of the closed loop that removes it (Theorem~\ref{thm:int}). Second, for positive systems, we obtain \emph{scalable} formulations: an SDP whose number of parameters grows linearly (rather than quadratically) in the state dimension (Theorem~\ref{thm:SDP:scalable}), and an exact LP for linear costs (Theorem~\ref{thm:LP:scalable}). Unlike \cite{nguyen2025scalable,anand2024scalable}, which achieve scalability by bounding the attack energy, our formulations require no such bound, which may not be known in general. Third, the analysis extends beyond LTI dynamics to polynomial systems via an SOS program (Theorem~\ref{thm:SOS}). Preliminary versions of isolated instances of \eqref{eq:main:sup} appeared in our earlier work for discrete-time and/or single-output settings \cite{teixeira2015strategic,nguyen2022single,anand2020joint}; the present paper provides a unified continuous-time, multiple-output treatment with complete proofs.

\begin{table*}[t]
\centering
\caption{Summary of results for different specifications of $\ell(\cdot,\cdot)$, $\Uc$, and $|\Mc|$ in \eqref{eq:main:sup}.}
\label{main:res:tab}
\begin{tabular}{|c|c|c|c|c|l|}
\hline
System class & $\ell(y,S)$ & $\Uc$ & $|\Mc|$ & Result & Program and guarantee \\ \hline\hline 
LTI & $y^\top S y$ & $\Lte$ & $1$ & Theorems~\ref{thm:LTI:L2e}, \ref{thm:inf:zero}, \ref{thm:int} & Exact SDP; unboundedness conditions; closed-loop remedy \\ \hline
LTI & $y^\top S y$ & $\Lte$ & $>1$ & Theorem~\ref{thm:LTI:lower} & SDP, upper bound; unboundedness conditions \\ \hline
LTI & $y^\top S y$ & $\Lt^0$  & $>1$ & Theorem~\ref{thm:LTI:L2} & SDP, upper bound; unboundedness conditions \\ \hline
Positive LTI & $y^\top S y$ & $\Lt^{0+}$ & $>1$ & Theorem~\ref{thm:SDP:scalable} & Scalable SDP ($O(n)$ parameters), upper bound \\ \hline
Positive LTI & $q^\top |y|$   & $\Lo^{+}$ & $1$ & Theorem~\ref{thm:LP:scalable} & Exact LP \\ \hline
Polynomial & $y^\top S y$ & $\Lte$ & $>1$ & Theorem~\ref{thm:SOS} & SOS program, upper bound \\ \hline
\end{tabular}
\end{table*}

The value of \eqref{eq:main:sup} depends strongly on the choice of the input signal space $\Uc$. The following example illustrates this by exhibiting a system whose worst-case cost is finite over $\Lt^{0}$ but unbounded over $\Lte$\footnote{For reproducibility, the code can be found at \url{https://tinyurl.com/fw9ktkak}}.
\begin{example}\label{exmp:QTP}
Consider the quadruple tank process \cite{johansson2002quadruple} where the states denote the height of the tanks, the inputs denote the pump flows into the top two tanks, and the measurements denote the height of the bottom two tanks. The adversary injects attacks into the actuator channels and aims to maintain the attacked outputs near the nominal output while simultaneously making the state values deviate far from zero. This is contrary to the operator objective which aims to maintain the nominal states. 

We first solve the optimization problem \eqref{eq:main:sup} when $\Uc = \Lt^{0}$, using Theorem~\ref{thm:LTI:L2}, and obtain the optimal value as $\gamma_{\Lt} = 19.6859$. The SDP (and all other examples in this paper) is solved using YALMIP \cite{lofberg2004yalmip} and MOSEK~11 \cite{aps2019mosek}. That is, for any square integrable attack signals which eventually stop, the maximum state deviation (measured in terms of $2$-norm) is at-most $\gamma_{\Lt}$. Next, we solve the optimization problem \eqref{eq:main:sup} when $\Uc = \Lte$, using Theorem~\ref{thm:LTI:L2e}, and obtain the optimal value as $\gamma_{\Lte} = +\infty$. That is, for an adversary that injects attack signals in the $\Lte$ space, the adversary can cause unbounded deviation in the states without being detected, which exposes a fundamental structural limitation in the system. 
$\hfill \triangleleft$
\end{example}
\textbf{Related works:}
In this section, we review prior works that study optimization problems of the form \eqref{eq:main:sup}. The work \cite{milovsevic2018quantifying} quantifies the worst-case $\ell_\infty$ norm of a performance measure under stealthy attack constraints, where an attack is defined as stealthy if it does not trigger an alarm at the detector. In contrast, the present paper focuses on performance measures based on the $L_2$ and $L_1$ norms. Related studies along these lines include \cite{sui2020vulnerability,milovsevic2019estimating}. 

The work \cite{khazraei2022resiliency} considers the maximum state deviation (measured in terms of the $L_2$ norm) caused by stealthy attacks on nonlinear systems, where stealthiness is defined as the ability to remain undetected under any possible detector. The work \cite{sargolzaei2021secure} studies the maximum asymptotic deviation caused by bounded attacks and disturbances. In contrast, our work focuses on polynomial systems and provides a constructive approach to quantify security against stealthy attacks. Related studies along these lines include \cite{shang2024nonlinear,kato2021security}

The work \cite{wang2007worst} designs an observer to maximize the ratio between the $H_\infty$ norm and the $H_{\_}$ index. When $m = p = 1$ and $N = 2$, the result in Theorem~\ref{thm:LTI:L2e} of this article coincides with this ratio. A detailed discussion of this equivalence is provided in \cite{teixeira2021security}.

The work \cite{yoon2005worst} studies a minimax optimal control problem for uncertain stochastic systems and designs control policies against worst-case uncertainties \cite[(14)]{yoon2005worst}. In \cite{elia2002minimization}, the authors propose a finite-horizon control policy that minimizes the worst-case $\ell_\infty$ norm of a performance measure under bounded disturbances. In contrast, the present paper considers performance criteria based on $L_2$ and $L_1$ norms rather than the $\ell_\infty$ norm. Moreover, the results obtained here are horizon length independent. 
 
The paper \cite{meng2017dynamic} considers the problem of designing control policies in cooperative-antagonistic networks where some agent pairs are cooperative and some are antagonistic. 
Related studies along these lines include \cite{ssun2017controllability,meng2019extended}. 
The paper \cite{lipp2016antagonistic} proposes a method to obtain the value of \eqref{eq:main:sup} in finite horizon. However, the results presented in this paper are horizon length independent, and scalable.

\textbf{Notation:}
A matrix $A \in \mathbb{R}^{n \times n}$ is called Metzler, denoted by $A \in \mathbb{M}^n$, if all its off-diagonal entries are nonnegative, i.e., $a_{ij} \geq 0$ for all $i \neq j$. 
Let $x: \mathbb{R}_+ \to \mathbb{R}^n$ be a CT signal. Then, the $2$-norm over the horizon $[0,N]$ is $\|x\|_{L_2,[0,N]}^2 = \int_0^{N} x(t)^\top x(t)\, dt$. The $1$-norm of $x$ over the horizon $[0,N]$ is $\|x\|_{L_1,[0,N]} = \int_0^{N} \mathbf{1}^\top |x(t)|\, dt$, where $|\cdot|: \mathbb{R}^n \to \mathbb{R}_+^n$ denotes the element-wise absolute value.
\begin{equation}
\begin{aligned}
\mathcal{L}_p &= \left\{ x: \mathbb{R}_+ \to \mathbb{R}^n \;\middle|\; \|x\|_{L_p,[0,\infty)} < \infty \right\}, \quad p \in \{1,2\}\\
\mathcal{L}_p^+ &= \left\{ x: \mathbb{R}_+ \to \mathbb{R}_+^n \;\middle|\; \|x\|_{L_p,[0,\infty)} < \infty \right\}\\
\mathcal{L}_p^0 &= \left\{ x \in \mathcal{L}_p \;\middle|\; \lim_{t \to \infty} x(t) = 0 \right\}\\
\mathcal{L}_p^{0+} &= \left\{ x \in \mathcal{L}_p^+ \;\middle|\; \lim_{t \to \infty} x(t) = 0 \right\}\\
\mathcal{L}_{pe} &= \left\{ x: \mathbb{R}_+ \to \mathbb{R}^n \;\middle|\; \|x\|_{L_p,[0,N]} < \infty, \;\forall N \in \mathbb{R}_+ \right\}\\
\mathcal{L}_{pe}^+ &= \left\{ x: \mathbb{R}_+ \to \mathbb{R}_+^n \;\middle|\; \|x\|_{L_p,[0,N]} < \infty, \;\forall N \in \mathbb{R}_{+} \right\},
\end{aligned}
\end{equation}
\begin{equation}
\begin{aligned}
\mathcal{L}_{pe}^{0} &= \left\{ x \in \mathcal{L}_{pe} \;\middle|\; \lim_{t \to \infty} x(t) = 0 \right\},\\
\mathcal{L}_{pe}^{0+} &= \left\{ x \in \mathcal{L}_{pe}^+ \;\middle|\; \lim_{t \to \infty} x(t) = 0  \right\}.
\end{aligned}
\end{equation}
\section{Linear time-invariant systems}\label{sec:LTI}
In this section, we consider in \eqref{P_main}, an LTI system 
\begin{equation}\label{eq:LTI}
\begin{aligned}
\dot{x}(t) &= Ax(t) + Bu(t), \\
y_i(t) &= C_ix(t) + D_iu(t)   
\end{aligned}
\end{equation}
where the matrices are of appropriate dimension. For simplicity, we denote $\Sigma_i = (A,B,C_i,D_i), i \in \{1,2,\dots,N\}$. We next establish the following. 
\begin{assumption}\label{ass:stable}
It holds that $\max_i \Re(\lambda_i(A)) < 0 \hfill \triangleleft$
\end{assumption}
\begin{assumption}\label{ass:ctrb}
The tuple $(A,B)$ is controllable. $\hfill \triangleleft$
\end{assumption}
Since quadratic functions are well studied in the literature for LTI systems, we adopt the following loss function in the remainder of this section $\ell(y,S) = y^\top S y$.
Then, we are interested in solving the optimization problem 
\begin{equation}\label{eq:main:LTI}
\begin{aligned}
    \sup_{u \in \mathcal{U}} \;\; 
    & \;\frac{1}{\lvert \mathcal{P} \rvert}
    \sum_{i \in \mathcal{P}} \left[ \int_{0}^{\infty} y_i(t)^\top Q_i y_i(t) \, dt \right] \\
    \text{s.t.} \;\;
    & \; \int_{0}^{\infty} y_j(t)^\top R_j y_j(t) \, dt \le \delta_j, 
    \quad \forall j \in \mathcal{M},\\
    & \; \text{System dynamics \eqref{eq:LTI}},
\end{aligned}
\end{equation}
Before we present the results of this section, we present some preliminaries. Given an output $y_i$, with $i \in \{1,\dots,N\}$, let $y_{ik}$, $k \in \{1,\dots,p_i\}$, denote the components of the output vector.
\begin{defn}[Finite zeros]
Given $\Sigma_i$, $\beta \in \mathbb{C}$ is an invariant zero of $\Sigma_i$ if $\exists\;x_0 \in \mathbb{R}^n, g \in \mathbb{R}^m$ such that 
\begin{equation}\label{eq:zero}
R(\beta,x_0,g,\Sigma_i) := 
    \begin{bmatrix}
        \beta I-A & B\\
        C_i & -D_i
    \end{bmatrix}\begin{bmatrix}
        x_0\\ g
    \end{bmatrix} = \begin{bmatrix}
        0 \\ 0
    \end{bmatrix}.
\end{equation}
The zero is defined to be unstable if $\Re(\beta) > 0. \hfill\triangleleft$
\end{defn}
\begin{defn}[Relative degree]
Consider a strictly proper SISO system $\Sigma$ with realization $(A, B, C, 0)$. Let 
$r = \min \left\{ k \in \mathbb{Z}_+ \;\middle|\; C A^{k-1} B \neq 0 \right\}$.
Then the system $\Sigma$ has relative degree $r$. $\hfill \triangleleft$
\end{defn}
\begin{theorem}\label{thm:LTI:L2e}
Suppose that $\mathcal{U} = \mathcal{L}_{2e}$ and $N=2$. Without loss of generality, let $\mathcal{P}=\{1\}$ and $\mathcal{M} = \{2\}$. Then the following statements hold:
\begin{enumerate}[label=(\alph*)]
\item The value of the optimization problem \eqref{eq:main:LTI} is equal to the value of the SDP
\begin{equation}\label{eq:main:SDP}
\begin{aligned}
\min_{\gamma, P} & \; \gamma \delta_2\\
\text{s.t.}&\; 
\begin{bmatrix}
A^\top P + PA & PB \\
B^\top P & 0
\end{bmatrix} + \Gamma \preceq 0,\\
&\; \gamma \ge 0,\; P \in \mathbb{R}^{n \times n},\; P = P^\top \succeq 0
\end{aligned}
\end{equation}
where\; $\Gamma \triangleq \begin{bmatrix}
    C_1^\top\\
    D_1^\top
\end{bmatrix}Q_1\begin{bmatrix}
    C_1 & D_1
\end{bmatrix}
- \gamma \begin{bmatrix}
    C_2^\top\\
    D_2^\top
\end{bmatrix}R_2\begin{bmatrix}
    C_2 & D_2
\end{bmatrix}$
\item The value of \eqref{eq:main:LTI} is unbounded if the unstable finite zeros of $\Sigma_2$ are not the zeros of $\Sigma_1$.
\end{enumerate}
\end{theorem}
\vspace{-20pt}
\begin{pf}
The optimization problem \eqref{eq:main:LTI} can be formulated as a convex dual
%
%
%
%
\begin{equation}\label{pf:0}
\begin{aligned}
&\inf_{\gamma \geq 0} \;\gamma \delta_2\\
&\text{s.t.} \;\scalebox{0.85}{$\displaystyle \int_{0}^{\infty} \left(  \gamma y_2^\top(t) R_2 y_2(t) - y_1^\top(t) Q_1 y_1(t) \right) dt \geq 0,\;\forall u \in \mathcal{L}_{2e}$}
\end{aligned} 
\end{equation}
Since $N=2$, the convex dual is equivalent to the primal problem \cite[Theorem 4.2.1]{petersen2012robust}.
The equivalence between the primal problem~\eqref{eq:main:LTI} and its dual~\eqref{pf:0} requires a strict feasibility condition: there must exist $u \in \mathcal{L}_{2e}$ that strictly satisfies the constraint. Taking $u = 0$ yields $x \equiv 0$ (since $x(0) = 0$), hence $\int_0^\infty y_2^\top R_2 y_2 \, dt = 0 < \delta_2$. Thus strong duality holds and~\eqref{pf:0} is equivalent to~\eqref{eq:main:LTI}.

Let us now define the supply rate as $s(u,t) \triangleq  \gamma y_2(t)^\top R_2 y_2(t) - y_1(t)^\top Q_1 y_1(t)$. Then, using the dissipation inequality \cite[Theorem 8.4.5]{trentelman1991dissipation}, it follows that \eqref{pf:0} is equivalent to \eqref{eq:main:SDP}. This completes the proof of part (a). For part (b), the proof follows the same approach as Theorem~\ref{thm:LTI:lower}(b) and is thus omitted. $\hfill \blacksquare$
\end{pf}
Theorem~\ref{thm:LTI:L2e} presents the first contribution of this paper and provides an exact SDP to determine the value of \eqref{eq:main:sup} when $N=2$. Note that Theorem~\ref{thm:LTI:L2e}(b) provides a sufficient condition for the optimization problem to be unbounded. However, the theorem addresses only finite zeros. We next provide a specific set of conditions based on zeros at infinity (relative degree) under which the optimization problem~\eqref{eq:main:SDP} becomes infeasible. 
\begin{theorem}\label{thm:inf:zero}
Suppose that $\mathcal{U} = \mathcal{L}_{2e}$ and $N = 2$. Without loss of generality, let $\mathcal{P} = \{1\}$ and $\Mc = \{2\}$. 
Assume that $m = 1$, $p_1 \ge p_2$, $D_1=0$ and $D_2=0$. 
Let $r_{ij}$, with $i \in \{1,2\}$ and $j \in \{1,\dots,p_i\}$, denote the relative degree of the system from the input to the output $y_{ij}$, and define
\begin{equation*}
r_1^{\min} = \min_{j \in \{1,\dots,p_1\}} r_{1j},
\qquad
r_2^{\min} = \min_{j \in \{1,\dots,p_2\}} r_{2j}.
\end{equation*}
Then the SDP~\eqref{eq:main:SDP} is infeasible if $r_1^{\min} < r_2^{\min}$.
\end{theorem}
\vspace{-20pt}
\begin{pf}
The optimization problem \eqref{eq:main:LTI} can be formulated as a convex dual
\begin{equation}\label{eq:pf:1}
\begin{aligned}
\inf_{\gamma \geq 0} & \; \gamma \delta_2\\
\text{s.t.} & \; \scalebox{0.8}{$\displaystyle\int_{0}^{\infty} x^\top(t) \left( C_1^\top Q_1 C_1 - \gamma C_2^\top R_2 C_2 \right) x(t) \, dt \leq 0, \forall u \in \mathcal{L}_{2e}$}
\end{aligned} 
\end{equation}
where the dynamics satisfy \eqref{eq:LTI}. For this special case with $N=2$, the dual problem \eqref{eq:pf:1} is equivalent to the primal problem \eqref{eq:main:LTI} \cite[Theorem 4.2.1]{petersen2012robust}. Observe that the left-hand side of the constraint in \eqref{eq:pf:1} becomes strictly positive when $\gamma =0$, which violates the constraint. Thus, for \eqref{eq:pf:1} to be feasible, it is necessary that $\gamma > 0$. Let us call this observation $O1$. A necessary condition for the constraint in \eqref{eq:pf:1} to hold is the Frequency Domain Inequality (FDI) \cite{willems1974existence}
\begin{equation}\label{eq:pf:2}
G_1(s)^H G_1(s) - \gamma G_2(s)^H G_2(s) \leq 0
\end{equation}
holds $\forall s \in \mathbb{C}$, $\Re(s) \geq 0$ where $G_1(s) = \begin{bmatrix}
    g_{11}(s) & \dots & g_{1p_1}(s)
\end{bmatrix}^\top$, $G_2(s) = \begin{bmatrix}
    g_{21}(s) & \dots & g_{2p_2}(s)
\end{bmatrix}^\top$.
%
For $s \in \mathbb{R}^+$
the necessary FDI \eqref{eq:pf:2} reduces to
\begin{equation}\label{eq:pf:3}
    \Phi(s) = \sum_{i=1}^{p_1} g_{1i}(s)^2 - \gamma \sum_{j=1}^{p_2} g_{2j}(s)^2 \leq 0, \quad \forall s \in \mathbb{R}^+.
\end{equation}
In the remainder of the proof we exhibit a single $\bar{s} \in \mathbb{R}^+$ such that $\Phi(s) > 0$ for all $s > \bar{s}$, contradicting \eqref{eq:pf:3} and hence the necessary condition for feasibility.

Since each $g_{ij}$ is a proper transfer function with relative degree $r_{ij}$, by the definition of relative degree there exist constants $K_{ij} \neq 0$ such that
\begin{equation}\label{eq:a}
    \lim_{s \to \infty} s^{r_{ij}} g_{ij}(s) = K_{ij}.
\end{equation}
Multiplying $\Phi(s)$ by $s^{2 r_1^{\min}} > 0$, define $\Psi(s) = $
\begin{equation}\label{eq:psi}
s^{2 r_1^{\min}} \Phi(s)
= \underbrace{\sum_{i=1}^{p_1} s^{2 r_1^{\min}} g_{1i}(s)^2}_{=:\, A(s)}
- \gamma \underbrace{\sum_{j=1}^{p_2} s^{2 r_1^{\min}} g_{2j}(s)^2}_{=:\, B(s)}.
\end{equation}
Consider first $A(s)$. For each $i \in \{1,\dots,p_1\}$, write
\begin{equation}
s^{2 r_1^{\min}} g_{1i}(s)^2
= s^{-2(r_{1i} - r_1^{\min})} \left( s^{r_{1i}} g_{1i}(s) \right)^2 .
\end{equation}
By \eqref{eq:a}, $\left( s^{r_{1i}} g_{1i}(s) \right)^2 \to K_{1i}^2$ as $s \to \infty$. The exponent $r_{1i} - r_1^{\min} \geq 0$, so $s^{-2(r_{1i} - r_1^{\min})} \to 1$ if $r_{1i} = r_1^{\min}$ and $s^{-2(r_{1i} - r_1^{\min})} \to 0$ if $r_{1i} > r_1^{\min}$. Then by product rule, 
\begin{equation}
\lim_{s \to \infty} s^{2 r_1^{\min}} g_{1i}(s)^2 =
\begin{cases}
K_{1i}^2, & r_{1i} = r_1^{\min},\\[2pt]
0, & r_{1i} > r_1^{\min}.
\end{cases}
\end{equation}
Summing the finitely many terms, we get $\lim_{s \to \infty} A(s) = \sum_{i \,:\, r_{1i} = r_1^{\min}} K_{1i}^2 >\; 0.$ Let this be observation $O2$. Consider now $B(s)$. For each $j \in \{1,\dots,p_2\}$, write
\begin{equation}
s^{2 r_1^{\min}} g_{2j}(s)^2
= s^{-2(r_{2j} - r_1^{\min})} \left( s^{r_{2j}} g_{2j}(s) \right)^2 .
\end{equation}
By hypothesis $r_{2j} \geq r_2^{\min} > r_1^{\min}$, so the exponent satisfies $r_{2j} - r_1^{\min} > 0$ and hence $s^{-2(r_{2j} - r_1^{\min})} \to 0$. By \eqref{eq:a}, $\left( s^{r_{2j}} g_{2j}(s) \right)^2 \to K_{2j}^2$, a finite constant. By the product rule for limits, each term tends to $0$, and summing the finitely many terms, we get $\lim_{s \to \infty} B(s) = 0$. Let this be observation $O3$.

By combining $O1$, $O2$ and $O3$ we obtain $\lim_{s \to \infty} \Psi(s) = \sum_{i \,:\, r_{1i} = r_1^{\min}} K_{1i}^2 > 0$. By the definition of the limit, there exists $\bar{s} \in \mathbb{R}^+$ such that $\Psi(s) > 0$, and therefore $\Phi(s) = s^{-2 r_1^{\min}} \Psi(s) > 0$, for all $s > \bar{s}$. This contradicts the necessary condition \eqref{eq:pf:3}. Since \eqref{eq:pf:3} is necessary for the feasibility of \eqref{eq:pf:1}, equivalently \eqref{eq:main:SDP} is infeasible. $\hfill \blacksquare$
\end{pf}
Theorem~\ref{thm:inf:zero} presents the second contribution of this paper and provides a sufficient condition, based on relative degrees, under which the optimization problem \eqref{eq:main:SDP} becomes infeasible. We next demonstrate the infeasibility with a numerical example.
\begin{example}\label{exmp:1}
Consider the stable dynamical systems:
\begin{equation}
    G_1(s) = \begin{bmatrix}
    \frac{1}{s+4}\\       
    \frac{1}{2s+3}
    \end{bmatrix},
    \quad
    G_2(s) = \begin{bmatrix}
        \frac{1}{s^2+4s+2}\\
        \frac{1}{(s+6)(s+7)(s+8)}
    \end{bmatrix}.
\end{equation}
with $Q_1=I_2$ and $R_2=I_2$. The relative degrees are $r_{11} = 1$, $r_{12} = 1$, $r_{21} = 2$, and $r_{22} = 3$. Since  ${r}_1^{\min} = 1< 2 = {r}_2^{\min}$, the conditions of Theorem~\ref{thm:inf:zero} are satisfied. Consequently, the optimization problem \eqref{eq:main:LTI} is unbounded. While \eqref{eq:main:LTI} can be formulated as the SDP \eqref{eq:main:SDP}, numerical solvers will either report infeasibility or encounter numerical difficulties.$\hfill \triangleleft$
\end{example}

Using the conditions in Theorem~\ref{thm:inf:zero}, the structural limitations of the system can be identified. However, once these conditions are rectified, let's say at the IT layer, a system operator might be interested in quantifying security. To this end, we propose a method in Theorem~\ref{thm:int}. In other words, we propose a constructive approach to modify the closed-loop dynamics so that the optimization problem \eqref{eq:main:LTI} becomes feasible.
\begin{theorem}\label{thm:int}
Suppose that $\mathcal{U} = \mathcal{L}_{2e}$ and $N = 2$. Without loss of generality, let $\mathcal{P} = \{1\}$ and $\mathcal{M} = \{2\}$. 
Assume that $m = 1$, $p_1 \ge p_2$, $D_1=0$ and $D_2=0$. 
Let $r_{ij}$, with $i \in \{1,2\}$ and $j \in \{1,\dots,p_i\}$, denote the relative degree of the system from the input to the output $y_{ij}$, and let $r_2^{\min} = \min_{j \in \{1,\dots,p_2\}} r_{2j}$.
Now consider the following modified closed-loop dynamics:
\begin{equation}\label{eq:CL:modify}
    \begin{aligned}
        \dot{\bar{x}}(t) &= \bar{A} \bar{x}(t) + \bar{B}u(t)\\
        \bar{y}_1(t) &= \bar{C}_1\bar{x}(t),\;
        \bar{y}_2(t) = \bar{C}_2\bar{x}(t),
    \end{aligned}
\end{equation}
where $\bar{y}_{2}(t) = y_{2}(t)$. For all $j \in \{1,\dots,p_1\}$, we define
\begin{equation}\label{eq:bar}
    \bar{y}_{1j}(t) = \begin{cases}
    y_{1j}(t) & \text{if } r_{1j} \geq r_2^{\min}\\[0.2cm]
    \underbrace{\int_0^t\!\cdots\!\int_0^{\tau_{1}}}_{\tilde{r}_{1j}\ \text{times}}
   y_{1j}(\tau_0)\, d\tau_0\,\cdots d\tau_{\tilde{r}_{1j}-1},&\text{otherwise,}
\end{cases}
\end{equation}
where $\tilde{r}_{1j} = r_2^{\min} - r_{1j}$. Let $\bar{r}_{ij}$, with $i \in \{1,2\}$ and $j \in \{1,\dots,p_i\}$, denote the relative degree of the modified closed-loop system from the input to the output $\bar{y}_{ij}$. Then for \eqref{eq:CL:modify}, it holds that $\bar{r}_1^{\min} \geq \bar{r}_2^{\min}$.
\end{theorem}
\vspace{-15pt}
\begin{pf}
We establish the proof by showing that the modified closed-loop system \eqref{eq:CL:modify} satisfies $\bar{r}_{1j} \geq r_2^{\min}$ for all $j \in \{1,\dots,p_1\}$. To this end, let us denote the transfer function from $u$ to $y_{ij}$ as
\begin{equation}
    g_{ij}(s) = \frac{b_{m_{ij}} s^{m_{ij}} + b_{m_{ij}-1}s^{m_{ij}-1} + \cdots + b_1 s + b_0}{s^{n_{ij}} + a_{n_{ij}-1}s^{n_{ij}-1} + \cdots + a_1 s + a_0},
\end{equation}
where $r_{ij} = n_{ij} - m_{ij}$ is the relative degree.
By construction, we have $\bar{y}_{2}(t) = y_{2}(t)$ for all $t \geq 0$. Consequently, we have $\bar{r}_{2j} = r_{2j}$ for all $j \in \{1,\dots,p_2\}$, and thus $\bar{r}_2^{\min} = r_2^{\min}$. For the outputs of subsystem 1, we distinguish between two cases:
\textit{Case 1: $r_{1j} \geq r_2^{\min}$.} Here, by construction, we have $\bar{y}_{1j}(t) = y_{1j}(t)$. Consequently, we have $\bar{r}_{1j} = r_{1j} \geq r_2^{\min}$.
\textit{Case 2: $r_{1j} < r_2^{\min}$.} In this case, the output is integrated $\tilde{r}_{1j} = r_2^{\min} - r_{1j}$ times. Applying $\tilde{r}_{1j}$ integrations sequentially results in multiplication of the transfer function by $\frac{1}{s^{\tilde{r}_{1j}}}$. Thus, the transfer function from $u$ to $\bar{y}_{1j}$ is $\bar{g}_{1j}(s)= \frac{1}{s^{\tilde{r}_{1j}}} \cdot g_{1j}(s)$.
The relative degree of this modified transfer function is
$\bar{r}_{1j} = (n_{1j} + \tilde{r}_{1j}) - m_{1j} = (n_{1j} - m_{1j}) + \tilde{r}_{1j} = r_{1j} + \tilde{r}_{1j}$. Substituting $\tilde{r}_{1j}$, we find $\bar{r}_{1j} = r_{1j} + (r_2^{\min} - r_{1j}) = r_2^{\min}$. Combining both cases, we conclude that $\bar{r}_{1j} \geq r_2^{\min}$ for all $j \in \{1,\dots,p_1\}$. Taking the minimum over $j$ yields $\bar{r}_1^{\min} \geq r_2^{\min} = \bar{r}_2^{\min}$, which concludes the proof. $\hfill \blacksquare$
\end{pf}

\begin{rem}
During implementation, the integration in \eqref{eq:CL:modify} can be approximated by a leaky integrator
to avoid stability issues. The relative degree calculation  holds identically for both ideal and leaky integrators.$\hfill \triangleleft$
\end{rem}


We next depict the efficacy of the closed-loop modification proposed in Theorem~\ref{thm:int}.
\begin{example}
Consider the system in Example~\ref{exmp:1}, and let $\delta_2=1$. Since the conditions of Theorem~\ref{thm:inf:zero} are satisfied, the optimal value of \eqref{eq:main:LTI} is unbounded. To make the value bounded for analysis purposes, we apply Theorem~\ref{thm:int}, which indicates that outputs $y_{11}$  and $y_{12}$ must be integrated once. In this example, we consider leaky integrators and modify the transfer function $G_1(s)$ as follows:
    \begin{equation}\label{eq:exmp:2}
    \bar{G}_1(s) = \begin{bmatrix}
    \frac{1}{(s+a_1)(s+4)} &       
    \frac{1}{(s+a_2)(s+a_3)(2s+3)}
    \end{bmatrix}^\top,
\end{equation}
where $a_i, i \in \{1,2,3\}$ is uniformly distributed in the range $[0.1,0.3]$. We then use Theorem~\ref{thm:LTI:L2e} to formulate an equivalent SDP to solve \eqref{eq:main:LTI} under the modified closed-loop dynamics \eqref{eq:exmp:2}, and obtain an optimal value of $28.2017$. $\hfill \triangleleft$
\end{example}

The results presented so far only hold when there is exactly one constraint. From a security standpoint, the result presented so far can be used to quantify attack impact when there is only one detection constraint. However, in large scale systems like power grids, there can be multiple detectors implying multiple constraints. Thus, we next provide two results to determine the value of \eqref{eq:main:LTI} when there are multiple constraints. In Theorem~\ref{thm:LTI:lower}, we provide a method to determine an upper bound of \eqref{eq:main:LTI} when $\mathcal{U} = \mathcal{L}_{2e}$, and in Theorem~\ref{thm:LTI:L2}, we provide a method to determine an upper bound of \eqref{eq:main:LTI} when $\mathcal{U} = \mathcal{L}_2^{0}$.

Before presenting the results, we introduce some notation. Let $\Sigma_p$ ($\Sigma_m$) denote the state space system with input $u$ and output vector $y_p$ ($y_m$), where
\begin{equation}\label{eq:ypym}
y_p = \frac{1}{\sqrt{\vert \mathcal{P} \vert}} \underset{i \in \mathcal{P}}{\operatorname{col}} \left( \sqrt{Q_i} y_i \right),\;
y_m = \underset{i \in \mathcal{M}}{\operatorname{col}}\left(\sqrt{R_j} y_j \right).
\end{equation}
\begin{theorem}\label{thm:LTI:lower}
Suppose that \;$\mathcal{U} = \mathcal{L}_{2e}$. Then the following statements hold: 
\begin{enumerate}[label=(\alph*)]
    \item The value of the optimization problem \eqref{eq:main:LTI} is upper bounded by the value of the SDP \vspace{-10pt}
\begin{adjustwidth}{-\leftmargin}{0pt}
\begin{equation}\label{eq:SDP:lower}
\begin{aligned}
    \min_{\bm{\gamma} \geq 0, P} & \;  \sum_{j \in \mathcal{M}} \gamma_j \delta_j \\
    \text{s.t.}&\; \begin{bmatrix}
        A^\top P + PA & PB \\
        B^\top P & 0
    \end{bmatrix} + \Gamma \preceq 0, \;P = P^\top \succeq 0 
\end{aligned}
\end{equation}
\end{adjustwidth}
where $\bm{\gamma} \in \mathbb{R}^{\vert \mathcal{M}\vert}$ denotes the stacked vector of $\gamma_j$, and \vspace{-15pt}
\begin{adjustwidth}{-\leftmargin}{0pt}
\begin{equation}
\Gamma = \frac{1}{\vert \mathcal{P} \vert}\sum_{i \in \Pc} \begin{bmatrix}
        C_i^\top\\
        D_i^\top
    \end{bmatrix}Q_i \begin{bmatrix}
        C_i & D_i
    \end{bmatrix}  - \sum_{j \in \Mc} \gamma_j\begin{bmatrix}
        C_j^\top\\
        D_j^\top
    \end{bmatrix}R_j \begin{bmatrix}
        C_j & D_j
    \end{bmatrix}  
\end{equation}
\end{adjustwidth}
\item The value of the optimization problem \eqref{eq:SDP:lower} is unbounded if $\Sigma_m$ has unstable finite zeros that are not zeros of $\Sigma_p$.
\end{enumerate}
\end{theorem}
\vspace{-15pt}
\begin{pf}
Consider the optimization problem \eqref{eq:main:LTI}
\begin{equation}\label{pf:a:0}
\begin{aligned}
    \sup_{u \in \mathcal{U}} \;\; 
    &\; \mathcal{F}_0(y) \triangleq \frac{1}{\lvert \mathcal{P} \rvert}
    \sum_{i \in \mathcal{P}} \int_{0}^{\infty} y_i(t)^\top Q_i y_i(t) \, dt \\ 
    \text{s.t.} \;\;
    & \; \int_{0}^{\infty} y_j(t)^\top R_j y_j(t) \, dt \le \delta_j, 
    \quad \forall j \in \mathcal{M}
\end{aligned}
\end{equation}
Using the epigraph formulation, \eqref{pf:a:0} can be reformulated as
\begin{equation}\label{pf:a:1}
\inf_{\beta \geq 0} \left\{  \beta \; \big| \; \mathcal{F}_0(y) \leq \beta,\; \forall u \in \mathcal{A}\right\}
\end{equation}
$\mathcal{A} = \left\{ u\in \Lte \,\Big|\, \displaystyle \int_{0}^{\infty} y_j(t)^\top R_j y_j(t) \, dt \leq \delta_j, \forall j \in \mathcal{M} \right\}$. The optimization problem \eqref{pf:a:1} can be interpreted as follows: at optimality, $\mathcal{F}_0(y) \leq \beta$ holds for all $u \in \mathcal{A}$. Let us call this condition $C1$.  

A sufficient condition for $C1$ is that there exists $\bm{\gamma} \geq 0$ such that  
\begin{equation}\label{pf:a:2}
 \beta \geq \mathcal{F}_0(y) + \sum_{j \in \mathcal{M}} \gamma_j \left( \delta_j - \int_{0}^{\infty} y_j(t)^\top R_j y_j(t) \, dt \right)
\end{equation}
holds for all $u \in \Lte$. To see that \eqref{pf:a:2} is sufficient for $C1$, note that whenever $u \in \mathcal{A}$, we have
\begin{equation}\label{pf:a:3} 
\delta_j - \int_{0}^{\infty} y_j(t)^\top R_j y_j(t) \, dt \geq 0, \quad \forall j \in \mathcal{M}
\end{equation}
Since $\bm{\gamma} \geq 0$, \eqref{pf:a:2} implies  
\begin{equation}
\beta - \mathcal{F}_0(y) \geq \sum_{j \in \mathcal{M}} \gamma_j \left( \delta_j - \int_{0}^{\infty} y_j(t)^\top R_j y_j(t) \, dt \right) \geq 0,
\end{equation}
holds $\forall u \in \mathcal{A}$. Therefore, $\mathcal{F}_0(y) \leq \beta$ for all $u \in \mathcal{A}$, as required. Thus, replacing the constraint in \eqref{pf:a:1} with the sufficient condition \eqref{pf:a:2} yields
\begin{equation}\label{pf:a:4}
\inf_{\beta \geq 0, \bm{\gamma}\geq 0} \;\left\{ \beta \; \Big|\; \eqref{pf:a:2} \;\text{holds}\; \forall u \in \Lte 
\right\}
\end{equation}
Since we replaced the original constraint in \eqref{pf:a:1} with a more restrictive sufficient condition in \eqref{pf:a:2}, the optimal value of \eqref{pf:a:4} provides an upper bound on the optimal value of \eqref{pf:a:1}. Let us now rewrite \eqref{pf:a:4} as 
\begin{equation}\label{pf:a:5}
\inf_{\bm{\gamma} \geq 0} \sup_{u \in \Lte} \mathcal{F}_0(y) + \sum_{j \in \mathcal{M}} \gamma_j \left( \delta_j -\int_{0}^{\infty} y_j(t)^\top R_j y_j(t) \, dt \right)
\end{equation}
Equivalently \eqref{pf:a:5} can be written as 
\begin{align}
&\; \inf_{\bm{\gamma} \geq 0} \Big\{  \sum_{j \in \mathcal{M}} \gamma_j \delta_j + \kappa\Big\} \label{pf:a:6}\\
& \kappa = \sup_{u \in \Lte} \mathcal{F}_0(y) - \sum_{j \in \mathcal{M}} \gamma_j \int_{0}^{\infty} y_j(t)^\top R_j y_j(t) \, dt
\end{align}
Observe that $\kappa$ is a maximization problem with a quadratic objective over the unbounded set $\Lte$. Since a quadratic functional can be made arbitrarily large unless it is non-positive everywhere, it follows that
\begin{equation}\label{pf:a:7}
    \kappa = \begin{cases}
        0 & \text{if}\; \mathcal{F}_0(y) - \displaystyle \sum_{j \in \mathcal{M}} \gamma_j \int_{0}^{\infty} y_j(t)^\top R_j y_j(t) \, dt   \leq 0\\
        \infty & \text{otherwise}
    \end{cases}
\end{equation}
Using \eqref{pf:a:7}, \eqref{pf:a:6} can be reformulated as
\begin{equation}\label{pf:a:8}
\begin{aligned}
\inf_{\bm{\gamma} \geq 0} &\; \sum_{j \in \mathcal{M}} \gamma_j \delta_j \\
\text{s.t.} 
&\; \mathcal{F}_0(y) - \displaystyle \sum_{j \in \mathcal{M}} \gamma_j \int_{0}^{\infty} y_j(t)^\top R_j y_j(t) \, dt \leq 0,
\end{aligned}
\end{equation}
Define the supply rate
\begin{equation}\label{eq:supply}
s(u, t) \triangleq \displaystyle \sum_{j \in \mathcal{M}} \gamma_j y_j(t)^\top R_j y_j(t) - \frac{1}{\lvert \mathcal{P} \rvert}
    \sum_{i \in \mathcal{P}}  y_i(t)^\top Q_i y_i(t).
\end{equation}
By linearity, \eqref{pf:a:8} can be rewritten as 
\begin{equation}\label{pf:a:9}
\inf_{\bm{\gamma} \geq 0} \Big\{ \sum_{j \in \mathcal{M}} \gamma_j \delta_j \; \big| \; \int_0^{\infty} s(u(t)) dt \geq 0,\; \forall u \in \Lte \Big\}.
\end{equation}
By dissipative system theory \cite[Theorem 8.4.5]{trentelman1991dissipation}, the constraint $\displaystyle\int_0^{\infty} s(u(t)) dt \geq 0$ is equivalent to the existence of a quadratic storage function satisfying the matrix inequality \eqref{eq:SDP:lower}. This completes the proof of $(a)$. 

Using dissipation inequalities \cite{willems1974existence}, a necessary condition for the constraint in \eqref{pf:a:9} to hold is that
\begin{equation}\label{pf:a:10}
G_p(s)^H G_p(s) - G_m(s)^H \Lambda(\bm{\gamma}) G_m(s) \leq 0,
\end{equation}
holds $\forall s \in \mathbb{C}, \Re(s) \geq 0$ where 
$G_p(s)$ ($G_m(s)$) is the transfer function corresponding to $\Sigma_p$ ($\Sigma_m$), and $ \displaystyle \Lambda(\bm{\gamma}) = \operatorname{diag}_{j \in \Mc} \left( I_{p_j}{\gamma_j} \right)$. Equivalently, the necessary condition \eqref{pf:a:10} can be reformulated as
\begin{equation}\label{pf:a:11}
x^H \left(G_p(s)^H G_p(s) - G_m(s)^H \Lambda(\bm{\gamma}) G_m(s)\right)x \leq 0,
\end{equation}
holds $\forall s : \Re(s) \geq 0, \; \forall x \in \mathbb{C}^m$. To prove part (b), we show that when there exists an unstable finite zero for $\Sigma_m$ that is not a zero of $\Sigma_p$, the condition \eqref{pf:a:10} cannot be satisfied for any finite $\bm{\gamma}$, implying the optimization is unbounded. By the definition of finite zeros, the assumption that $\Sigma_m$ has a finite zero at $s_0$ that is not a zero of $\Sigma_p$ means there exists $s_0 \in \mathbb{C}$ with $\Re{(s_0)} >0$ and a nonzero vector $x_0 \in \mathbb{C}^m$ such that $G_m(s_0)x_0 =0$ and $G_p(s_0)x_0 \neq 0$. Substituting $s=s_0$ and $x=x_0$ into \eqref{pf:a:11}, the left-hand side becomes
\begin{equation}\label{pf:a:12}
 \underbrace{x_0^H G_p(s_0)^H G_p(s_0)x_0}_{> 0} - \underbrace{x_0^HG_m(s_0)^H \Lambda(\bm{\gamma}) G_m(s_0)x_0}_{=0} > 0
\end{equation}
which contradicts \eqref{pf:a:10}. Therefore, no finite $\bm{\gamma}$ can satisfy the constraint, and the optimization problem \eqref{eq:SDP:lower} is unbounded. This completes the proof. $\hfill \blacksquare$
\end{pf}
We next depict the results presented in Theorem~\ref{thm:LTI:lower} with a numerical example.
\begin{example}\label{exmp:upper}
Consider the LTI system \eqref{eq:LTI} where 
\begin{equation}
    A = \begin{bmatrix}
        0 & 1 \\
        0 & \varepsilon 
    \end{bmatrix},
    B=\begin{bmatrix}
        0\\ 1
    \end{bmatrix},
    C_1= \begin{bmatrix}
        1\\5
    \end{bmatrix}^\top, C_2= e_2^\top, C_3= e_1^\top
\end{equation}
with $\delta_2 = \delta_3 =1$, $R_2 = R_3 = 1$, $Q_1=1$, and $e_j$ denoting the $j$-th standard basis vector in $\mathbb{R}^2$. Let $\varepsilon=-1$ which makes the system controllable and asymptotically stable. Solving the SDP \eqref{eq:SDP:lower} yields an optimal value of $36$. $\hfill \triangleleft$
\end{example}
Next, we provide a SDP to determine the upper bound of \eqref{eq:main:LTI} when $\mathcal{U} = \Lt^0$ and $N >1$.
\begin{theorem}\label{thm:LTI:L2}
Suppose that $\Uc = \Lt^{0}$. Assume that $D_i = 0,\;\forall i \in \{1,\dots,N\}$. Then the following statements hold: 
\begin{enumerate}[label=(\alph*)]
    \item The value of the optimization problem \eqref{eq:main:LTI} is upper bounded by the value of the SDP
\begin{equation}\label{eq:SDP:L2}
\begin{aligned}
    \min_{\bm{\gamma},P} & \;  \sum_{j \in \mathcal{M}} \gamma_j \delta_j \\
    \text{s.t.}&\; \begin{bmatrix}
        A^\top P + PA + \Gamma & PB \\
        B^\top P & 0
    \end{bmatrix} \preceq 0, \;\bm{\gamma} \geq 0,\\
    & \; \Gamma = \frac{1}{\vert \mathcal{P} \vert}\sum_{i \in \Pc} C_i^\top Q_i C_i - \sum_{j \in \Mc} \gamma_j C_j^\top R_j C_j  
\end{aligned}
\end{equation}
where $P = P^\top$ and $\bm{\gamma} \in \mathbb{R}^{\vert \mathcal{M}\vert}$ denotes the vertically stacked vector of $\gamma_j$.
\item The value of the optimization problem \eqref{eq:SDP:L2} is unbounded if $\Sigma_m$ has zeros on the imaginary axis that are not zeros of $\Sigma_p$.
\end{enumerate}
\end{theorem}
\vspace{-15pt}
\begin{pf}
Under the assumption $\Uc = \Lt^{0}$, the optimization problem \eqref{eq:main:LTI} can be reformulated as
\begin{equation}\label{pf:s:5}
\inf_{\beta \geq 0} \left\{  \beta \; \big| \; \mathcal{F}_0(y) \leq \beta,\; \forall u \in \mathcal{A}\right\}
\end{equation}
\begin{equation}
\text{where} \;\mathcal{F}_0(y) \triangleq \frac{1}{\lvert \mathcal{P} \rvert}
    \sum_{i \in \mathcal{P}} \int_{0}^{\infty} y_i(t)^\top Q_i y_i(t) \, dt,\\
\end{equation}
\begin{equation}
    \mathcal{A} = \Big\{ u\in \Lt^{0} \,\big|\, \int_{0}^{\infty} y_j(t)^\top R_j y_j(t) \, dt \leq \delta_j, \; \forall j \in \mathcal{M} \Big\}.
\end{equation}
At optimality, $\mathcal{F}_0(y) \leq \beta$ holds for all admissible inputs $u \in \mathcal{A}$. Next, we claim that the following two conditions are equivalent:
\begin{enumerate}[label=(C\arabic*)]
    \item $\mathcal{F}_0(y) \leq \beta$ holds for all $u \in \mathcal{A}$.
    \item There exists $\bm{\gamma} \geq 0$ such that \eqref{pf:s:6} holds $\forall u \in \Lt^{0}.$
\end{enumerate}
    \begin{equation}\label{pf:s:6}
    \mathcal{F}_0(y) + \sum_{j \in \Mc} \gamma_j \left( \delta_j - \int_{0}^{\infty} y_j(t)^\top R_j y_j(t) \, dt\right) \leq \beta
    \end{equation}
In other words, we claim that the constraint in \eqref{pf:s:5} can be equivalently replaced with its dual constraint (C2). Next, we prove this equivalence. We have shown in Theorem~\ref{thm:LTI:lower} that a sufficient condition for (C1) is (C2). We next show that (C1) $\Rightarrow$ (C2) using \cite[Theorem 4.3.2]{petersen2012robust}.

Let us consider the system \eqref{eq:LTI} with the stacked output vector $y \in \mathbb{R}^{\sum_{i \in \{1,\dots, N\}} p_i}$. Without loss of generality, we assume that the output vector is stacked as 
\begin{equation}
    y(t) = \begin{bmatrix}
        \operatorname{col}_{i \in \mathcal{P}} \left( \sqrt{Q_i} y_i \right)^\top & 
        \operatorname{col}_{j \in \mathcal{M}} \left(\sqrt{R_j} y_j \right)^\top
    \end{bmatrix}^\top.
\end{equation}
In other words, we assume that the first $\vert \Pc\vert$ block outputs affect the objective function and the remaining outputs affect the constraints. Let us now define
\begin{equation}\label{pf:s:G2}
\begin{aligned}
\mathcal{G}_0(y(\cdot)) &= \int_0^\infty y(t)^\top M_0 y(t)\,dt - \beta, \\
\mathcal{G}_j(y(\cdot)) &= \int_0^\infty y(t)^\top M_j y(t)\,dt + \delta_j, \quad \forall j \in \Mc, 
\end{aligned}
\end{equation}
where $M_0 = \operatorname{diag}_{i \in \mathcal{P}} \left( \frac{1}{\vert \Pc \vert} I_{p_i} \right)$ and $M_j = -I_{p_j}, \forall j \in \Mc$. Now suppose that (C1) holds. That is, 
$\mathcal{F}_0(y) - \beta = \mathcal{G}_0(y(\cdot)) \leq 0$ for all $u \in \Lt^{0}$ that satisfies $
\mathcal{G}_j(y(\cdot)) \geq 0$ for all $j \in \Mc$. Thus, for the functionals defined in \eqref{pf:s:G2}, condition (i) in \cite[Theorem 4.3.2]{petersen2012robust} holds. 

Assume that $u = 0 \in \Lt^{0}$. Then, since $x(0)=0$, the outputs are identically zero. It follows that $\mathcal{G}_j(y(\cdot)) = \delta_j  > 0$ for all $j \in \Mc$. Thus, for the functionals defined in \eqref{pf:s:G2}, condition (ii) in \cite[Theorem 4.3.2]{petersen2012robust} holds. Therefore, by \cite[Theorem 4.3.2]{petersen2012robust}, it follows that \eqref{pf:s:6} holds. 
%
%

We have now shown that (C1) and (C2) are equivalent, allowing us to replace the constraint in \eqref{pf:s:5} with the dual constraint. To complete the proof, we must show that constraint \eqref{pf:s:6} can be reformulated as the LMI in \eqref{eq:SDP:L2}. The key difference from Theorem~\ref{thm:LTI:lower} is that here $\Uc = \Lt^{0}$, requiring us to use the cyclo-dissipativity formulation from \cite{trentelman1991dissipation}. The details are omitted as they closely parallel the proof of Theorem~\ref{thm:LTI:lower}

We now explain why the value of \eqref{eq:SDP:L2} provides an upper bound for the dual problem. The cyclo-dissipativity result in \cite{trentelman1991dissipation} establishes that when $\Uc \subseteq \Lte$ and $\lim_{t \to \infty} x(t) = 0$, the dual constraint in \eqref{pf:s:6} is equivalent to the LMI constraint in \eqref{eq:SDP:L2}. In our setting, however, we restrict to $\Uc = \Lt^{0} \subset \Lte$. The LMI constraint in \eqref{eq:SDP:L2} is equivalent for requiring the dual constraint \eqref{pf:s:6} to hold for all $u \in \Lte$, whereas the original dual constraint \eqref{pf:s:6} only needs to hold for $u \in \Lt^{0}$. Since the LMI imposes a stronger requirement (validity over a larger input class), any feasible solution to \eqref{eq:SDP:L2} is also feasible for the dual problem with $\Uc = \Lt^{0}$. Therefore, the optimal value \eqref{eq:SDP:L2} upper bounds the optimal value of the dual problem, which concludes the proof of (a).

Using dissipation inequalities \cite[Theorem 8.4.3]{trentelman1991dissipation}, a necessary and sufficient condition for the LMI constraint in \eqref{eq:SDP:L2} to hold is that 
\begin{equation}\label{pf:s:10}
G_p(i\omega)^H G_p(i\omega) - G_m(i\omega)^H \Lambda(\bm{\gamma}) G_m(i\omega) \preceq 0
\end{equation}
holds $\forall \; \omega \in \mathbb{R}$ and $G_p(s)$ ($G_m(s)$) is the transfer function corresponding to $\Sigma_p$ ($\Sigma_m$), and $\Lambda(\bm{\gamma}) = \operatorname{diag}_{j \in \Mc} \left( I_{p_j}{\gamma_j} \right)$. Then, \eqref{pf:s:10} can be reformulated as
\begin{equation}\label{pf:ee:10}
x^H \left(
G_p(i\omega)^H G_p(i\omega) - G_m(i\omega)^H \Lambda(\bm{\gamma}) G_m(i\omega)
\right)x \leq 0
\end{equation}
holds $\forall \omega \in \mathbb{R}, x \in \mathbb{C}^m$. Thus, \eqref{eq:SDP:L2} can be rewritten as
\begin{equation}\label{lem_s2}
    \inf_{\bm{\gamma}}\; \; \Big\{ \sum_{j \in \Mc} \gamma_j \delta_j \big| \;\eqref{pf:ee:10} \;\text{holds} \; \forall \omega \in \mathbb{R}, x \in \mathbb{C}^m  \Big\}\\
\end{equation}
Assume that there exists a bounded $\Lambda(\gamma)$ that solves \eqref{lem_s2}. We also assume that there exists a real number $\omega$ on the imaginary axis which is a zero of the system ${\Sigma}_{m}$ but not a zero of ${\Sigma}_{p}$. By the definition of a zero, there exists $g \neq 0$ such that ${G}_{m}(i \omega)g=0$ and ${G}_{p}( i \omega)g \neq 0$. Under these assumptions, the constraint of \eqref{lem_s2} can be rewritten as $-g^H {G}_{p}(i \omega)^H{G}_{p}(i \omega)g \geq 0$, which cannot hold since ${G}_{p}(i \omega)g \neq 0$. That is, the feasibility set of \eqref{lem_s2} is empty, which contradicts our assumption. This concludes the proof.  $\hfill \blacksquare$
\end{pf}
%
\begin{example}
    Consider the dynamical system in Example~\ref{exmp:upper}. We let $\varepsilon=-1$ which makes the system controllable and asymptotically stable. Solving the SDP \eqref{eq:SDP:L2} yields an optimal value of $26$. $\hfill \triangleleft$
\end{example}
\section{Positive linear time-invariant systems}\label{sec:positive}
The results presented up to now can be effectively used to determine the value of the optimization problem \eqref{eq:main:sup} for small-scale systems. However SDPs 
scale badly and inhibit its use for large-scale systems 
\cite[Chapter 11.8.3]{boyd2004convex}.
To this end, 
we adopt the positive system framework which has proven to be practically relevant \cite{farina2011positive} and scalable \cite{rantzer2015scalable}. We begin this section by presenting the definition of a positive system. 
\begin{definition}[Positive system]\label{defn:positive}
    A CT dynamical system $\Sigma = (A,B,C,D)$ is internally positive if $A \in \mathbb{M}^n$, $B \geq 0$, $C \geq 0$, and $D \geq 0$. $\hfill \triangleleft$
\end{definition}


Thus in this section, we restrict our attention to dynamical systems of the form \eqref{P_main} satisfying positivity, as defined in Definition~\ref{defn:positive}. In other words, for LTI positive systems considered in this paper, the inputs, outputs and the states remain positive at all times. 
For positive systems, we present the results for both (a) quadratic costs in Section~\ref{sec:pos:quadratic} and (b) linear costs in Section~\ref{sec:pos:linear}.
\subsection{Quadratic cost and constraints}\label{sec:pos:quadratic}
In this subsection, we are then interested in solving the optimization problem of the form
\begin{equation}\label{opt:positive:L2:primal}
\begin{aligned}
    \sup_{u \in \Lt^{0+}} \;\; 
    & \;\frac{1}{\lvert \mathcal{P} \rvert}
    \sum_{i \in \mathcal{P}} \left[ \int_{0}^{\infty} y_i(t)^\top Q_i y_i(t) \, dt \right] \\
    \text{s.t.} \;\;
    & \; \int_{0}^{\infty} y_j(t)^\top R_j y_j(t) \, dt \le \delta_j
    \quad \forall j \in \mathcal{M},\\
    &\; \text{Positive system dynamics \eqref{eq:LTI}}.
\end{aligned}
\end{equation}
For positive systems, the inputs by adversaries should also remain positive to remain stealthy. If the input becomes negative, and the closed loop system is at equilibrium $x^* =0$, the output might become negative which can trivially be detected as an anomaly by the operator. Thus, in this subsection, we restrict our attention to inputs in the signal space $u \in \Lt^{0+}$.
\begin{theorem}\label{thm:SDP:scalable}
Consider a positive LTI system \eqref{eq:LTI} satisfying Definition~\ref{defn:positive}. Then the following statements hold:
\begin{enumerate}[label=(\alph*)]
    \item The value of the optimization problem \eqref{opt:positive:L2:primal} is upper bounded by the value of the SDP\vspace{-10pt}
\begin{adjustwidth}{-\leftmargin}{0pt}
    \begin{equation}\label{eq:SDP:scalable}
    \begin{aligned}
    \min_{\bm{\gamma}, P} & \;  \sum_{j \in \mathcal{M}} \gamma_j \delta_j \\
    \text{s.t.}&\; \begin{bmatrix}
        A^\top P + PA  & PB \\
        B^\top P & 0
    \end{bmatrix}+ \Gamma \preceq 0, \;\bm{\gamma} \geq 0,\\
    &\; P \succeq 0, \Gamma_{ij} \geq 0 \;\; \forall\, i \neq j, \Gamma_{ii} \geq 0 \;\; \forall\, i \leq n.
\end{aligned}
\end{equation}
  \end{adjustwidth}
where $P$ is a diagonal matrix, $\geq$ denotes an element-wise non-negativity constraint, and 
\begin{equation}
\scalebox{0.9}{$
\Gamma = \frac{1}{\vert \mathcal{P} \vert} \underset{i \in \Pc}{\sum} \begin{bmatrix}
        C_i^\top\\
        D_i^\top
    \end{bmatrix}Q_i \begin{bmatrix}
        C_i & D_i
    \end{bmatrix} -  \underset{j \in \Mc}{\sum} \gamma_j\begin{bmatrix}
        C_j^\top\\
        D_j^\top
    \end{bmatrix}R_j \begin{bmatrix}
        C_j & D_j
    \end{bmatrix}$}
\end{equation}
\item The value of the optimization problem \eqref{eq:SDP:scalable} is unbounded if $\Sigma_m$ has zeros on the imaginary axis that are not zeros of $\Sigma_p$.
\end{enumerate}
\end{theorem}
\vspace{-20pt}
\begin{pf}
The optimization problem \eqref{opt:positive:L2:primal} can be formulated as a convex dual
\begin{equation}\label{eq:pos:1}
\inf_{\bm{\gamma} \geq 0} \Big\{ \sum_{j \in \mathcal{M}} \gamma_j \delta_j \; \big| \; \int_0^{\infty} s(u(t)) \, dt \geq 0,\; \forall u \in \mathcal{L}_2^{0+} \Big\}
\end{equation}
The dual problem \eqref{eq:pos:1} provides an upper bound for \eqref{opt:positive:L2:primal}, which can be established by arguments analogous to those in Theorem~\ref{thm:LTI:lower}\footnote{Here, we are unable to establish strong duality as it requires the absence of direct feedthrough terms.}. Using dissipation inequalities \cite[Theorem 8.4.3]{trentelman1991dissipation}, a sufficient condition for the constraint in \eqref{eq:pos:1} to hold is that 
\begin{equation}\label{eq:pos:2}
G_p(i\omega)^H G_p(i\omega) - G_m(i\omega)^H \Lambda(\bm{\gamma}) G_m(i\omega) \preceq 0,
\end{equation}
holds $\forall \omega \in \mathbb{R}$, where $G_p(s)$ ($G_m(s)$) is the transfer function corresponding to $\Sigma_p$ ($\Sigma_m$), and $\Lambda(\bm{\gamma}) = \operatorname{diag}_{j \in \Mc} \left( I_{p_j}{\gamma_j} \right)$. Using the definitions of $\Sigma_p$ and $\Sigma_m$ in \eqref{eq:ypym}, we get 
\begin{equation}\label{eq:pos:3}
\begin{aligned}
    G_p(i\omega) &= \frac{1}{\sqrt{\vert \mathcal{P} \vert}} \operatorname{col}_{i \in \mathcal{P}} \left( \sqrt{Q_i} \begin{bmatrix}
        C_i & D_i
    \end{bmatrix}\right) \Xi(\omega),\\
    G_m(i\omega) &= \operatorname{col}_{j \in \mathcal{M}} \left( \sqrt{R_j} \begin{bmatrix}
        C_j & D_j
    \end{bmatrix}\right) \Xi(\omega),
\end{aligned}
\end{equation}
where $\Xi(\omega) = \begin{bmatrix}
        (i \omega I - A)^{-1}B\\
        I_{m}
    \end{bmatrix}$. Using the definitions from \eqref{eq:pos:3}, \eqref{eq:pos:2} can be equivalently reformulated as $\Xi(\omega)^H\;
 \Gamma \; \Xi(\omega) \preceq 0$, holds $ \forall \omega \in [0,\infty)$, where $\Gamma$ is defined in the theorem statement. We now impose the condition that $\Gamma_{ij} \geq 0 \; \forall\, i \neq j$, $\Gamma_{ii} \geq 0 \;\; \forall\, i \leq n$ which we refer to as A1. 

From \cite[Theorem~1]{rantzer2015kalman}, when A1 is satisfied, condition $\Xi(\omega)^H\;\Gamma \; \Xi(\omega) \preceq 0$ is equivalent to the SDP constraint in \eqref{eq:SDP:scalable}. Since the constraint $\Gamma_{ij} \geq 0 \;\; \forall\, i \neq j, \Gamma_{ii} \geq 0 \;\; \forall\, i \leq n$ is explicitly included in the theorem statement, A1 holds by construction, which concludes the proof of part~(a). The proof of part~(b) follows analogously to the sufficiency proof of part~(b) in Theorem~\ref{thm:LTI:L2} and is therefore omitted. $\hfill \blacksquare$
\end{pf}
Note that the SDPs presented in the previous sections (for instance, Theorem~\ref{thm:LTI:L2}) involve a symmetric full matrix variable $P$. Consequently, the number of decision variables is $n(n+1)/2$, resulting in quadratic growth with respect to the number of states. In this section, we propose an alternative upper bound based on a diagonal matrix $P$ with only $n$ parameters. This diagonal parameterization reduces the computational complexity to linear growth in $n$, enabling scalable analysis for large-scale systems. The trade-off of this scalable approach is that we restrict our analysis to the influence of positive attack signals on positive systems. As discussed previously, this restriction is realistic for many physical processes with nonnegative states. We next illustrate the results of Theorem~\ref{thm:SDP:scalable} through a numerical example.
\begin{example}
Consider a LTI positive system where
\begin{equation}
    A = \begin{bmatrix}
        -2 &   0 &  3\\
      0 &   -3 &   1\\
      0 &    0 &  -5
    \end{bmatrix}, B = \begin{bmatrix}
        1\\1\\1
    \end{bmatrix}, C_2 = \begin{bmatrix}
        I_2 & 0\\
        0 & 0
    \end{bmatrix}, D_2 = \begin{bmatrix}
        0\\0\\1
    \end{bmatrix}
\end{equation}
$C_1 = 10I_3$, $D_1 =0_{3 \times 1}$, $\delta_2=1$, $Q_1=I_3$, and $R_2=I_3$. We solve the optimization problem \eqref{eq:SDP:scalable} and obtain the value of $46.67$. In other words, the adversary, by injecting only positive attack signals, can at-most, cause a performance loss of $46.67$ units. $\hfill \triangleleft$
\end{example}
\subsection{Linear cost and constraints}\label{sec:pos:linear}
In this paper, we have thus far considered optimization problems of the form \eqref{eq:main:sup} with quadratic costs of the form $\ell(y,S) = y^\top S y$, $y \in \mathbb{R}^p$, $S \in \mathbb{R}^{p \times p}$. In this subsection, we consider optimization problems of the form \eqref{eq:main:sup} with linear costs of the form $\ell(y,S) = S^\top |y|$, $y \in \mathbb{R}^p$, $S \in \mathbb{R}^{p}$, $S \geq 0,$
where $\geq 0$ denotes element-wise inequality and $|\cdot|: \mathbb{R}^p \to \mathbb{R}_+^p$ denotes the element-wise absolute value operator. We remark that studying linear costs for positive systems has been a growing area of interest \cite{chen2013,gurpegui2023minimax,ohlin2024optimal}. However, such costs have not yet been used to study the antagonistic control problem, which we address.

In particular, we consider LTI positive systems with linear costs. The dissipativity result for positive systems established in \cite{haddad2005stability} guarantees the existence of a storage function that allows us to transform the infinite-dimensional optimization problem into a finite-dimensional one. This enables us to formulate an equivalent linear program (LP). Specifically, we aim to solve the optimization problem
%
\begin{equation}\label{opt:positive:L1:primal}
\scalebox{0.9}{$ \displaystyle
\sup_{u \in \mathcal{L}_{1e}^+} \Bigg\{ \int_{0}^{\infty} q^\top |y_{1}(t)| \; dt\;\Bigg\vert\;
\begin{matrix}
\displaystyle \int_{0}^{\infty} r^\top |y_{2}(t)| \; dt \le \delta,\; \\
\text{Positive dynamics \eqref{eq:LTI}}
\end{matrix}
\Bigg\}$}
\end{equation}
where $q \in \mathbb{R}^{p_1}$, $q \geq 0$, $r \in \mathbb{R}^{p_2}$ and $r \geq 0$. Next, we present the main result of this section. 
\begin{theorem}\label{thm:LP:scalable}
Suppose that the value of the optimization problem 
\eqref{opt:positive:L1:primal} is finite. Then it is equal to the value 
of the following LP
\begin{equation}\label{LP:pos}
\min_{\gamma \geq 0, \lambda \in \mathbb{R}^n, \lambda \geq 0} \Bigg\{ \delta \gamma \;\;\Big\vert \;\;\begin{matrix}
    {A}^\top \lambda + C_1^\top q - \gamma C_2^\top r \leq 0\\
 B^\top \lambda + D_1^\top q - \gamma D_2^\top r \leq 0
\end{matrix}
\Bigg\}
\end{equation}
\end{theorem}
\vspace{-20pt}
\begin{pf}
Since the system is positive and $u \in \mathcal{L}_{1e}^{+}$,  the absolute value operators in \eqref{opt:positive:L1:primal} can be dropped, and the objective and constraint are linear functionals of $u$. Furthermore, $\mathcal{L}_{1e}^{+}$ is a convex subset of the vector space $\mathcal{L}_{1e}$, and Slater's condition holds since $u = 0 \in \mathcal{L}_{1e}^{+}$ yields $\int_0^\infty r^\top y_2(t)\, dt = 0 < \delta$. Therefore, by Lagrange duality \cite[Chapter~8.6]{luenberger1997optimization}, and by arguments analogous to \eqref{pf:a:6}-\eqref{pf:a:7}, the optimization problem \eqref{opt:positive:L1:primal} is equivalent to
\begin{equation}\label{z1}
\scalebox{0.9}{$\min_{\gamma \geq 0} \Big\{ \delta \gamma \;\big\vert\;
\displaystyle \int_{0}^{\infty} \big[ q^\top y_1(t) - \gamma\, r^\top y_2(t) \big] 
\, dt \leq 0,\;\forall u \Big\}$}
\end{equation}
Applying the dissipativity results for positive systems from 
\cite[Theorem~6.2]{haddad2005stability} with supply rate $-q^\top y_1 + \gamma r^\top y_2$, the constraint in \eqref{z1} is equivalent to the existence of $\lambda \geq 0$ satisfying the inequalities in \eqref{LP:pos}. $\hfill \blacksquare$
\end{pf}
Theorem~\ref{thm:LP:scalable} provides an LP to compute the value of \eqref{opt:positive:L1:primal} when there is one constraint. Extending the results to multiple constraints is left for future work. 
We next depict the result of Theorem~\ref{thm:LP:scalable} using a numerical example. 
\begin{example}
Consider a LTI positive system where
$A = \begin{bmatrix}
        -0.1 &   0.1\\
      0.1 &   -0.1
    \end{bmatrix}$,
$B = I_2$, 
$C_1 = \begin{bmatrix}1 & 0 \end{bmatrix}$, $C_2 = \begin{bmatrix}0 & 1\end{bmatrix}$,
$D_1=0$, $D_2=0.5$, $r=q=1$ and $\delta=1$. We solve the optimization problem \eqref{LP:pos} and obtain the value of $20$. In other words, the adversary, by injecting only positive attack signals, can at-most, cause a performance loss of $20$ units. $\hfill \triangleleft$
\end{example}
\section{Polynomial systems}\label{sec:poly}
The results presented so far in this article, only hold when the process dynamics are LTI. We relax this assumption in this section and consider a nonlinear polynomial system of the form \eqref{P_main} where
\begin{equation}\label{eq:poly}
\begin{aligned}
    \dot{x}(t) &=  AZ(x(t)) + Bu(t),\; x(0)=0\\
    y_i(t) &= C_iZ(x(t)), \;i\in \{1,2,\dots,N\}
\end{aligned}
\end{equation}
where $Z(x)$ is a polynomial basis function in $x$, and $A$ and $B$ are matrices of appropriate dimension. Thus, in this section, we aim to solve 
\begin{equation}\label{opt:poly:primal}
\begin{aligned}
    \sup_{u \in \Lte} \;\; 
    & \;\frac{1}{\lvert \mathcal{P} \rvert}
    \sum_{i \in \mathcal{P}} \left[ \int_{0}^{\infty} y_i(t)^\top Q_i y_i(t) \, dt \right] \\
    \text{s.t.} \;\;
    & \; \int_{0}^{\infty} y_j(t)^\top R_j y_j(t) \, dt \le \delta_j
    \quad \forall j \in \mathcal{M},\\
    &\; \text{System dynamics \eqref{eq:poly}}
\end{aligned}
\end{equation}
Next we present the first main result of this section.
\begin{theorem}\label{thm:SOS}
The value of the optimization problem \eqref{opt:poly:primal} is upper bounded by the value of the SOS program
\begin{equation}\label{eq:SOS}
\begin{aligned}
     \underset{\gamma, S}{\min} & \; \sum_{j \in \Mc} \gamma_j \delta_j,\; \gamma \geq 0\\
     \text{s.t.} &\;  Z(x)^\top \Gamma Z(x)- \frac{\partial S(x)}{\partial x} (AZ(x) + Bu) \;\;\text{is SOS in} \; (x,u)\\
    &\; \Gamma = \sum_{j \in \Mc} C_j^\top \gamma_j R_j C_j  -
     \frac{1}{\lvert \mathcal{P} \rvert}
    \sum_{i \in \mathcal{P}}  C_i^\top Q_i C_i
\end{aligned}
\end{equation}
where $S(x)$ is a SOS polynomial in $x$.
\end{theorem}
\vspace{-15pt}
\begin{pf}
Using the arguments from the proof of Theorem~\ref{thm:LTI:lower}, it follows that the value of \eqref{opt:poly:primal} is upper bounded by the value of its dual counterpart
\begin{equation}\label{eq:nl:1}
\inf_{\bm{\gamma} \geq 0} \Big\{ \sum_{j \in \mathcal{M}} \gamma_j \delta_j \; \big| \; \int_0^{\infty} s(u(t)) dt \geq 0,\; \forall u \in \Lte \Big\},
\end{equation}
\begin{equation}
s(u(t)) \triangleq \displaystyle \sum_{j \in \mathcal{M}} \gamma_j y_j(t)^\top R_j y_j(t) - \frac{1}{\lvert \mathcal{P} \rvert}
    \sum_{i \in \mathcal{P}}  y_i(t)^\top Q_i y_i(t).
\end{equation}
Note that the proof follows directly from the proof of Theorem~\ref{thm:LTI:lower}, as it does not rely on the specific LTI structure of the dynamics. Consequently, the upper bound established for the dual problem in the LTI case extends to the polynomial systems considered here.

By dissipative system theory \cite{willems1972dissipative}, the constraint $\int_0^{\infty} s(u,t) dt \geq 0$ is equivalent to the existence of a storage function $S(x)$ such that $0 \leq s(u) - \left( \frac{\partial S(x)}{\partial x} \times  \dot{x} \right), \forall (x,u)$.
Similar to \cite{papachristodoulou2005tutorial}, let us consider a SOS polynomial storage function $S(x)$, such that $S(0)=0$. Then the condition $0 \leq s(u) - \left( \frac{\partial S(x)}{\partial x} \times  \dot{x} \right)$ is equivalent to the constraint in \eqref{eq:SOS}, which concludes the proof. $\hfill \blacksquare$
\end{pf}
Theorem~\ref{thm:SOS} establishes that the optimal value of the optimization problem is upper bounded by the value of the SOS program in \eqref{eq:SOS}. This upper bound arises from two sources. First, there may exist a duality gap between the primal problem \eqref{opt:poly:primal} and its dual formulation \eqref{eq:nl:1}. Second, we restrict our search to polynomial storage functions of the form $S(x)$, whereas other nonlinear storage functions may potentially yield tighter bounds. We next depict the efficacy of the proposed SOS program through a numerical example. 
\begin{example}
Consider a polynomial system of the form 
\begin{equation}
    \begin{bmatrix}
        \dot{x}_1(t)\\
        \dot{x}_2(t)
        \end{bmatrix} = \begin{bmatrix}
            -1 & 1 & 0.1 & 0\\
            0 & -2 & 0 & 0.05
        \end{bmatrix} Z(x) + \begin{bmatrix}
            1\\
            0.5
        \end{bmatrix} u(t),
\end{equation}
where $Z(x) = \begin{bmatrix}
            x_1(t) &
            x_2(t) &
            x_1(t)x_2(t)&
            x_1(t)^2
        \end{bmatrix}^\top$, $y_1(t) = x_1(t)$, and $y_2(t) = x_2(t)$. We set $N=2$, $Q_1 = 1$, $R_2 = 1$, $\delta_2 = 1$, and compute an upper bound for the optimization problem \eqref{opt:poly:primal} using the method described in Theorem~\ref{thm:SOS}. A degree-8 polynomial storage function yields an upper bound of $0.2507$. Increasing the polynomial degree may potentially yield tighter bounds, albeit at greater computational expense. For example, a degree-26 storage function produces an improved bound of $0.2501$. $\hfill \triangleleft$
\end{example}
\section{Conclusion}
In this paper, we studied the antagonistic control problem, in which a 
constrained input seeks to maximize the average cost of a set of outputs while remaining bounded in terms of other outputs. For LTI systems, we provided an exact SDP when there is a single constraint and SDPs computing upper bounds when there are multiple constraints, together with sufficient conditions, in terms of system zeros and relative degrees, under which the worst-case cost is unbounded. For positive systems, we derived scalable formulations whose complexity grows linearly in the state dimension: a scalable SDP for quadratic costs and an exact LP for linear costs. Finally, we extended the analysis to polynomial systems via an SOS program. Future work includes synthesis of controllers that minimize the worst-case cost.
\bibliographystyle{plain}       
\bibliography{autosam} 

@article{milovsevic2019estimating,
  title={Estimating the impact of cyber-attack strategies for stochastic networked control systems},
  author={Milo{\v{s}}evi{\'c}, Jezdimir and Sandberg, Henrik and Johansson, Karl Henrik},
  journal={IEEE Transactions on Control of Network Systems},
  volume={7},
  number={2},
  pages={747--757},
  year={2019},
  publisher={IEEE}
}

@article{gurpegui2023minimax,
  title={Minimax linear optimal control of positive systems},
  author={Gurpegui, Alba and Tegling, Emma and Rantzer, Anders},
  journal={IEEE Control Systems Letters},
  volume={7},
  pages={3920--3925},
  year={2023},
  publisher={IEEE}
}

@inproceedings{papachristodoulou2005tutorial,
  title={A tutorial on sum of squares techniques for systems analysis},
  author={Papachristodoulou, Antonis and Prajna, Stephen},
  booktitle={Proceedings of the 2005, American Control Conference, 2005.},
  pages={2686--2700},
  year={2005},
  organization={IEEE}
}

@article{willems1972dissipative,
  title={Dissipative dynamical systems part I: General theory},
  author={Willems, Jan C},
  journal={Archive for rational mechanics and analysis},
  volume={45},
  number={5},
  pages={321--351},
  year={1972},
  publisher={Springer}
}

@article{rantzer2015scalable,
  title={Scalable control of positive systems},
  author={Rantzer, Anders},
  journal={European Journal of Control},
  volume={24},
  pages={72--80},
  year={2015},
  publisher={Elsevier}
}

@article{sui2020vulnerability,
  title={The vulnerability of cyber-physical system under stealthy attacks},
  author={Sui, Tianju and Mo, Yilin and Marelli, Dami{\'a}n and Sun, Ximing and Fu, Minyue},
  journal={IEEE Transactions on Automatic Control},
  volume={66},
  number={2},
  pages={637--650},
  year={2020},
  publisher={IEEE}
}

@ARTICLE{nguyen2025scalable,
  author={Nguyen, Anh Tung and Anand, Sribalaji C. and Teixeira, André M. H.},
  journal={IEEE Transactions on Automatic Control}, 
  title={Scalable and Optimal Security Allocation in Networks Against Stealthy Injection Attacks}, 
  year={2026},
  volume={71},
  number={5},
  pages={3455-3462}
}

@inproceedings{anand2024scalable,
  title={Scalable Metrics to Quantify Security of Large-scale Systems},
  author={Anand, Sribalaji C and Grussler, Christian and Teixeira, Andr{\'e} MH},
  booktitle={2024 IEEE 63rd Conference on Decision and Control (CDC)},
  pages={7624--7630},
  year={2024},
  organization={IEEE}
}

@article{haddad2005stability,
  title={Stability and dissipativity theory for nonnegative dynamical systems: a unified analysis framework for biological and physiological systems},
  author={Haddad, Wassim M and Chellaboina, VijaySekhar},
  journal={Nonlinear Analysis: Real World Applications},
  volume={6},
  number={1},
  pages={35--65},
  year={2005},
  publisher={Elsevier}
}

@article{rantzer2015kalman,
  title={On the Kalman-Yakubovich-Popov lemma for positive systems},
  author={Rantzer, Anders},
  journal={IEEE Transactions on Automatic Control},
  volume={61},
  number={5},
  pages={1346--1349},
  year={2015},
  publisher={IEEE}
}

@book{farina2011positive,
  title={Positive linear systems: theory and applications},
  author={Farina, Lorenzo and Rinaldi, Sergio},
  year={2011},
  publisher={John Wiley \& Sons}
}

@inproceedings{milovsevic2018quantifying,
  title={Quantifying the impact of cyber-attack strategies for control systems equipped with an anomaly detector},
  author={Milo{\v{s}}evi{\'c}, Jezdimir and Umsonst, David and Sandberg, Henrik and Johansson, Karl Henrik},
  booktitle={2018 European Control Conference (ECC)},
  pages={331--337},
  year={2018},
  organization={IEEE}
}

@article{wang2007worst,
  title={Worst-case fault detection observer design: Optimization approach},
  author={Wang, HB and Wang, JL and Lam, J},
  journal={Journal of Opt. Theory and Applications},
  volume={132},
  number={3},
  pages={475--491},
  year={2007},
  publisher={Springer}
}

@article{elia2002minimization,
  title={Minimization of the worst case peak-to-peak gain via dynamic programming: state feedback case},
  author={Elia, Nicola and Dahleh, Munther A},
  journal={IEEE Transactions on Automatic Control},
  volume={45},
  number={4},
  pages={687--701},
  year={2002},
  publisher={IEEE}
}

@book{boyd2004convex,
  title={Convex optimization},
  author={Boyd, Stephen and Vandenberghe, Lieven},
  year={2004},
  publisher={Cambridge university press}
}

@article{yoon2005worst,
  title={On the worst-case disturbance of minimax optimal control},
  author={Yoon, Myung-Gon and Ugrinovskii, Valery A and Petersen, Ian R},
  journal={Automatica},
  volume={41},
  number={5},
  pages={847--855},
  year={2005},
  publisher={Elsevier}
}

@book{luenberger1997optimization,
  title={Optimization by vector space methods},
  author={Luenberger, David G},
  year={1997},
  publisher={John Wiley \& Sons}
}

@article{johansson2002quadruple,
  title={The quadruple-tank process: A multivariable laboratory process with an adjustable zero},
  author={Johansson, Karl Henrik},
  journal={IEEE Trans. on Control Systems Tech.},
  volume={8},
  number={3},
  pages={456--465},
  year={2002},
  publisher={IEEE}
}

@article{anand2020joint,
  title={Joint controller and detector design against data injection attacks on actuators},
  author={Anand, Sribalaji C and Teixeira, Andr{\'e} MH},
  journal={IFAC-PapersOnLine},
  volume={53},
  number={2},
  pages={7439--7445},
  year={2020},
  publisher={Elsevier}
}

@incollection{teixeira2021security,
  title={Security metrics for control systems},
  author={Teixeira, Andr{\'e} MH},
  booktitle={Safety, Security and Privacy for Cyber-Physical Systems},
  pages={99--121},
  year={2021},
  publisher={Springer}
}

@article{nguyen2022single,
  title={A single-adversary-single-detector zero-sum game in networked control systems},
  author={Nguyen, Anh Tung and Teixeira, Andr{\'e} MH and Medvedev, Alexander},
  journal={IFAC-PapersOnLine},
  volume={55},
  number={13},
  pages={49--54},
  year={2022},
  publisher={Elsevier}
}

@article{meng2017dynamic,
  title={Dynamic distributed control for networks with cooperative--antagonistic interactions},
  author={Meng, Deyuan},
  journal={IEEE Transactions on Automatic Control},
  volume={63},
  number={8},
  pages={2311--2326},
  year={2017},
  publisher={IEEE}
}

@article{ssun2017controllability,
  title={Controllability of multiagent networks with antagonistic interactions},
  author={Sun, Chao and Hu, Guoqiang and Xie, Lihua},
  journal={IEEE transactions on automatic control},
  volume={62},
  number={10},
  pages={5457--5462},
  year={2017},
  publisher={IEEE}
}

@article{meng2019extended,
  title={Extended structural balance theory and method for cooperative--antagonistic networks},
  author={Meng, Deyuan and Du, Mingjun and Wu, Yuxin},
  journal={IEEE Transactions on Automatic Control},
  volume={65},
  number={5},
  pages={2147--2154},
  year={2019},
  publisher={IEEE}
}

@article{willems1974existence,
  title={On the existence of a nonpositive solution to the Riccati equation},
  author={Willems, Jan},
  journal={IEEE Transactions on Automatic Control},
  volume={19},
  number={5},
  pages={592--593},
  year={1974},
  publisher={IEEE}
}

@incollection{trentelman1991dissipation,
  title={The dissipation inequality and the algebraic Riccati equation},
  author={Trentelman, Harry L and Willems, Jan C},
  booktitle={The Riccati Equation},
  pages={197--242},
  year={1991},
  publisher={Springer}
}

@book{petersen2012robust,
  title={Robust Control Design Using $H_{\infty}$ Methods},
  author={Petersen, Ian R and Ugrinovskii, Valery A and Savkin, Andrey V},
  year={2012},
  publisher={Springer Science \& Business Media}
}

@article{lipp2016antagonistic,
  title={Antagonistic control},
  author={Lipp, Thomas and Boyd, Stephen},
  journal={Systems \& Control Letters},
  volume={98},
  pages={44--48},
  year={2016},
  publisher={Elsevier}
}

@inproceedings{teixeira2015strategic,
  title={Strategic stealthy attacks: the output-to-output $\ell_2$-gain},
  author={Teixeira, Andr{\'e} and Sandberg, Henrik and Johansson, Karl H},
  booktitle={2015 54th IEEE Conference on Decision and Control (CDC)},
  pages={2582--2587},
  year={2015},
  organization={IEEE}
}

@article{mengis2017data,
  title={Data injection attacks on electricity markets by limited adversaries: Worst-case robustness},
  author={Mengis, Mateo R and Tajer, Ali},
  journal={IEEE Transactions on Smart Grid},
  volume={9},
  number={6},
  pages={5710--5720},
  year={2017},
  publisher={IEEE}
}

@inproceedings{lofberg2004yalmip,
  title={{YALMIP}: A toolbox for modeling and optimization in {MATLAB}},
  author={Lofberg, Johan},
  booktitle={2004 IEEE international conference on robotics and automation (IEEE Cat. No. 04CH37508)},
  pages={284--289},
  year={2004},
  organization={IEEE}
}

@article{aps2019mosek,
  title={Mosek optimization toolbox for matlab},
  author={ApS, Mosek},
  journal={User’s Guide and Reference Manual, Version},
  volume={4},
  number={1},
  pages={116},
  year={2019},
  publisher={MOSEK}
}

@article{ohlin2024optimal,
  title={Optimal control of linear cost networks},
  author={Ohlin, David and Tegling, Emma and Rantzer, Anders},
  journal={European Journal of Control},
  volume={80},
  pages={101068},
  year={2024},
  publisher={Elsevier}
}

@article{chen2013,
  title={$\ell_1$-induced norm and controller synthesis of positive systems},
  author={Chen, Xiaoming and Lam, James and Li, Ping and Shu, Zhan},
  journal={Automatica},
  volume={49},
  number={5},
  pages={1377--1385},
  year={2013},
  publisher={Elsevier}
}

@inproceedings{sargolzaei2021secure,
  title={A secure control design for networked control system with nonlinear dynamics under false-data-injection attacks},
  author={Sargolzaei, Arman},
  booktitle={2021 American Control Conference (ACC)},
  pages={2693--2699},
  year={2021},
  organization={IEEE}
}

@inproceedings{khazraei2022resiliency,
  title={Resiliency of nonlinear control systems to stealthy sensor attacks},
  author={Khazraei, Amir and Pajic, Miroslav},
  booktitle={2022 IEEE 61st Conference on Decision and Control (CDC)},
  pages={7109--7114},
  year={2022},
  organization={IEEE}
}

@article{shang2024nonlinear,
  title={Nonlinear stealthy attacks on remote state estimation},
  author={Shang, Jun and Zhou, Jing and Chen, Tongwen},
  journal={Automatica},
  volume={167},
  pages={111747},
  year={2024},
  publisher={Elsevier}
}

@article{kato2021security,
  title={Security analysis of linearization for nonlinear networked control systems under DoS},
  author={Kato, Rui and Cetinkaya, Ahmet and Ishii, Hideaki},
  journal={IEEE Transactions on Control of Network Systems},
  volume={8},
  number={4},
  pages={1692--1704},
  year={2021},
  publisher={IEEE}
}
\end{document}